\documentclass[11pt]{article}

\title{Elementary spectral invariants and quantitative closing lemmas for contact three-manifolds}
\author{Michael Hutchings\footnote{Partially supported by NSF grant DMS-2005437.}}
\date{}

\usepackage{amssymb}
\usepackage{latexsym}
\usepackage{amsmath}
\usepackage{amsthm}
\usepackage{amscd}
\usepackage{color}

\newcommand{\mc}[1]{{\mathcal #1}}

\numberwithin{equation}{section}

\newtheorem{theorem}{Theorem}[section]
\newtheorem{theoremstar}[theorem]{Theorem*}
\newtheorem{theoremstarstar}[theorem]{Theorem**}
\newtheorem{proposition}[theorem]{Proposition}

\newtheorem{lemma}[theorem]{Lemma}
\newtheorem{lemma-definition}[theorem]{Lemma-Definition}

\theoremstyle{definition}
\newtheorem{definition}[theorem]{Definition}

\newtheorem{remark}[theorem]{Remark}

\newtheorem{example}[theorem]{Example}
\newtheorem{examples}[theorem]{Examples}
\newtheorem{convention}[theorem]{Convention}

\newcommand{\floor}[1]{\left\lfloor #1 \right\rfloor}

\newcommand{\eqdef}{\;{:=}\;}

\newcommand{\C}{{\mathbb C}}

\newcommand{\R}{{\mathbb R}}

\newcommand{\Z}{{\mathbb Z}}

\newcommand{\op}{\operatorname}

\newcommand{\Ker}{\op{Ker}}

\newcommand{\bpm}{\begin{pmatrix}}
\newcommand{\epm}{\end{pmatrix}}

\renewcommand{\epsilon}{\varepsilon}

\begin{document}

\maketitle

\begin{abstract}
In a previous paper, we defined an ``elementary'' alternative to the ECH capacities of symplectic four-manifolds, using max-min energy of holomorphic curves subject to point constraints, and avoiding the use of Seiberg-Witten theory. In the present paper we use a variant of this construction to define an alternative to the ECH spectrum of a contact three-manifold. The alternative spectrum has applications to Reeb dynamics in three dimensions. In particular, we adapt ideas from a previous joint paper with Edtmair to obtain quantitative closing lemmas for Reeb vector fields in three dimensions. For the example of an irrational ellipsoid, we obtain a sharp quantitative closing lemma.
\end{abstract}

\tableofcontents

\setcounter{tocdepth}{2}

\section{Introduction}
\label{sec:intro}

\subsection{Quantitative closing lemmas}
\label{sec:qclintro}

Let $Y$ be a closed oriented $3$-manifold. Recall that a {\em contact form\/} on $Y$ is a $1$-form $\lambda$ with $\lambda\wedge d\lambda > 0$ everywhere. Associated to $\lambda$ is the {\em Reeb vector field\/} $R$ characterized by $d\lambda(R,\cdot) = 0$ and $\lambda(R)=1$. A {\em Reeb orbit\/} is a periodic orbit of $R$, that is a map $\gamma:\R/T\Z\to Y$ for some $T>0$, modulo translations of the domain, such that $\gamma'(t)=R(\gamma(t))$. The Reeb orbit $\gamma$ is {\em simple\/} if $\gamma$ is an embedding; otherwise $\gamma$ is the $d$-fold cover of a simple Reeb orbit for  some integer $d>1$. We define the {\em (symplectic) action\/} $\mc{A}(\gamma)>0$ to be the period $T$.

Closing lemmas in this context are concerned with creating Reeb orbits via suitable modifications of the contact form. A recent breakthrough in this topic is due to Irie \cite{irie} who proved, in the language of his later paper \cite{iriestrong}, that every contact form on a closed three-manifold has the following ``strong closing property'': If $\mc{U}\subset Y$ is a nonempty open set, and if $f:Y\to\R$ is a nonnegative smooth function supported in $\mc{U}$ which does not vanish identically, then for some $\tau\in[0,1]$, the contact form $e^{\tau f}$ has a Reeb orbit intersecting $\mc{U}$. A remarkable feature of this result is that, unlike in the older $C^1$ closing lemma of Pugh \cite{pugh}, no special care is required in the choice of $f$ to obtain a Reeb orbit. Note also that the statement implies that one can find $\tau$ arbitrary small such that $e^{\tau f}$ has a Reeb orbit intersecting $\mc{U}$, although the action of such a Reeb orbit might be very large.

In this paper we prove quantitative refinements of the strong closing property. Roughly speaking, these are answers to the following question: Given a nonempty open set $\mc{U}\subset Y$ and given $L>0$, how much do we need to deform the contact form $\lambda$ in $\mc{U}$ in order to produce a Reeb orbit intersecting $\mc{U}$ with action at most $L$? To make precise sense of this question, we need the following definitions.

\begin{definition}
\label{def:posdef}
Let $(Y,\lambda)$ be a closed contact three-manifold and let $\mc{U}\subset Y$ be a nonempty open set. A {\em positive deformation\/} of $\lambda$ supported in $\mc{U}$ is a smooth one-parameter family of contact forms
\[
\left\{\lambda_\tau=e^{f_\tau}\lambda\right\}_{\tau\in[0,1]}
\]
with the following properties:
\begin{itemize}
\item
$f_0\equiv 0$.
\item
$f_\tau(y)=0$ for all $\tau\in[0,1]$ and $y\in Y\setminus\mc{U}$.
\item
$f_1\ge 0$, and $f_1$ does not vanish identically.
\end{itemize}
\end{definition}

We now introduce a notion of the ``size'' of a positive deformation. Recall that the {\em symplectization\/} of $(Y,\lambda)$ is the symplectic 4-manifold
\[
\left(\R\times Y, \omega = d\left(e^s\lambda\right)\right)
\]
where $s$ denotes the $\R$ coordinate. Recall also that if $(X,\omega)$ is any symplectic 4-manifold, then its {\em Gromov width\/}
\[
c_{\op{Gr}}(X,\omega)\in (0,\infty]
\]
is defined to the supremum of $a>0$ such that there exists a symplectic embedding
\[
B^4(a)\longrightarrow (X,\omega).
\]
Here $B^4(a)$ denotes the ball
\[
B^4(a) = \left\{z\in\C^2 \;\big|\; \pi|z|^2 \le a \right\}
\]
with the restriction of the standard symplectic form $\sum_{i=1}^2dx_i\,dy_i$ on $\R^4=\C^2$. 

\begin{definition}
Let $\left\{\lambda_\tau = e^{f_\tau}\lambda\right\}_{\tau\in[0,1]}$ be a positive deformation as in Definition~\ref{def:posdef}. 
Let
\[
M_{\lambda_1} = \left\{(s,y) \in\R\times Y \mid 0 < s < f_1(y)\right\}
\]
with the symplectic form $d\left(e^s\lambda\right)$ from the symplectization. Define the {\em width\/} of the positive deformation by
\[
\op{width}(\{\lambda_\tau\}) = c_{\op{Gr}}(M_{\lambda_1}) \in (0,\infty).
\]
\end{definition}

\begin{definition}
Let $(Y,\lambda)$ be a closed contact $3$-manifold and let $L>0$. Define
\[
\op{Close}^L(Y,\lambda) \in [0,\infty]
\]
to be the infimum of $\delta>0$ with the following property:
\begin{itemize}
\item
Let $\mc{U}\subset Y$ be a nonempty set and let $\{\lambda_\tau\}_{\tau\in[0,1]}$ be a positive deformation of $\lambda$ supported in $\mc{U}$ with $\op{width}(\{\lambda_\tau\})\ge \delta$. Then for some $\tau\in[0,1]$, the contact from $\lambda_\tau$ has a Reeb orbit intersecting $\mc{U}$ with action at most $L$.
\end{itemize}
\end{definition}

\begin{remark}
\label{rem:closezeroinfinity}
It is immediate from the definition that $\op{Close}^L(Y,\lambda)$ is a nonincreasing function of $L$.

If $L$ is less than the minimum action of a Reeb orbit of $(Y,\lambda)$, then $\op{Close}^L(Y,\lambda)=+\infty$. In this case, for any $\delta>0$, one can prove that $\op{Close}^L(Y,\lambda)>\delta$ by taking $\mc{U}=Y$ and $\lambda_\tau = e^{r\tau}\lambda$ where $r>0$ is sufficiently large with respect to $\delta$.

It is also possible to have $\op{Close}^L(Y,\lambda)=0$. This is equivalent to the statement that every point in $Y$ is contained in a Reeb orbit of action at most $L$.
\end{remark}

\begin{remark}
If $\op{Close}^L(Y,\lambda)$ is finite, then for any nonempty open set $\mc{U}$, and any positive deformation $\{\lambda_\tau\}$ supported in $\mc{U}$ with width at least $\op{Close}^L(Y,\lambda)$, there must exist $\tau\in[0,1]$ such that the contact from $\lambda_\tau$ has a Reeb orbit intersecting $\mc{U}$ with action at most $L$. This Reeb orbit could, for example, be obtained by ``closing up'' a trajectory of the Reeb vector field in $Y\setminus \overline{\mc{U}}$ from $\partial \mc{U}$ to itself, or a cycle of such trajectories. If there is no such trajectory taking time less than $\op{Close}^L(Y,\lambda)$, then the Reeb orbit created must be entirely contained in $\overline{\mc{U}}$. Compare the proof of Proposition~\ref{prop:BoxClose} below.
\end{remark}

The following theorem, proved in \S\ref{sec:qcl}, asserts that for any $(Y,\lambda)$, if $L$ is sufficiently large then $\op{Close}^L(Y,\lambda)$ is finite, in fact $O(L^{-1})$. This is what we mean by a ``quantitative closing lemma''. To state the theorem, recall that the contact volume is defined by
\[
\op{vol}(Y,\lambda) = \int_Y\lambda\wedge d\lambda.
\]

\begin{theoremstar}[general quantitative closing lemma]
\label{theoremstar:qcl}
Let $(Y,\lambda)$ be a closed contact three-manifold. Then
\[
\limsup_{L\to\infty}\left(L\cdot\op{Close}^L(Y,\lambda)\right) \le \op{vol}(Y,\lambda).
\]
\end{theoremstar}

\begin{convention}
In this paper we are avoiding using Seiberg-Witten theory when we can. To keep track of this, results that currently require Seiberg-Witten theory for their proof are marked with an asterisk. Results that currently also require Seiberg-Witten theory to know that their statement makes sense are marked with two asterisks. 
\end{convention}

\begin{example}
Theorem*~\ref{theoremstar:qcl} implies that in the setting of Irie's strong closing property above, given $\delta\in(0,1]$, for some $\tau\in[0,\delta]$, a Reeb orbit intersecting $\mc{U}$ must appear with action $\mc{O}(\delta^{-1})$. This is because one can show that if $\delta>0$ is small then the positive deformation $\{e^{\tau\delta f}\lambda\}$ has width at least $c\delta$, where $c>0$ is a constant depending on $f$. Compare \cite[Prop.\ 6.2]{pfh4}.
\end{example}

Here is another general statement one can make. Let $Y\subset\R^4$ be a compact star-shaped (i.e.\ transverse to the radial vector field) hypersurface. Then the standard Liouville form
\[
\lambda_0 = \frac{1}{2}\sum_{i=1}^2\left(x_i\,dy_i - y_i\,dx_i\right)
\]
restricts to a contact form on $Y$, which we omit from the notation.

\begin{theorem}
[proved in \S\ref{sec:starshaped}]
\label{thm:ballgap}
Let $Y$ be a compact star-shaped hypersurface in $\R^4$ and suppose that $Y\subset B^4(a)$. Then for every $L>0$ we have
\[
\op{Close}^L(Y) \le \frac{2a}{\floor{La^{-1}}+3}.
\]
\end{theorem}

In specific examples one can say more if one can compute the alternative spectral invariants introduced in \S\ref{sec:ase} below. Here is an example where we can determine $\op{Close}^L(Y,\lambda)$ exactly. Let $a>1$ be irrational and consider the ellipsoid
\[
E(a,1) =\left\{z\in\C^2 \;\bigg|\; \frac{\pi|z_1|^2}{a }+ \pi|z_2|^2 \le 1\right\}.
\]
Then $\partial E(a,1)$ is a star-shaped hypersurface. The following result shows that $\op{Close}^L(\partial E(a,1))$ is related to how well $a$ can be approximated by rational numbers.

Fix $L\ge a$. Let $m_-,n_-$ be relatively prime integers with $m_->0$ such that $n_-/m_-$ is maximized subject to the constraints  $n_-/m_-<a$ and $am_-\le L$. Similarly, let $m_+,n_+$ be relatively prime integers with $m_+>0$ such that $n_+/m_+$ is minimized subject to the constraints  $n_+/m_+>a$ and $n_+\le L$.

\begin{theorem}[proved in \S\ref{sec:ellipsoid}]
\label{thm:ellipsoid}
If $a>1$ is irrational and $L\ge a$, then with the notation as above, we have
\[
\op{Close}^L(\partial E(a,1)) = \min(am_--n_-, \; n_+-am_+).
\]
\end{theorem}

Note that if $a>1$ is rational and $L\ge a$, then every point in $\partial E(a,1)$ is on a Reeb orbit of action $\le a$, so it follows from Remark~\ref{rem:closezeroinfinity} that $\op{Close}^L(\partial E(a,1))=0$.

\subsection{Obstructions to thickened Reeb trajectories}

We now discuss an application of quantitative closing lemmas, to show that long Reeb trajectories must come close to self-intersecting; more precisely, there are upper bounds on how much they can be ``thickened'' in the following sense.

\begin{definition}
Fix a closed contact three-manifold $(Y,\lambda)$. Given $A,L>0$, a {\em thickened Reeb trajectory\/} in $(Y,\lambda)$ of area $A$ and length $L$ is a smooth embedding
\[
\varphi:[0,L]\times D \longrightarrow Y,
\]
where $D$ is a closed disk of area $A$, such that
\[
\varphi^*\lambda = dt + \frac{1}{2}r^2\,d\theta.
\]
Here $t$ denotes the $[0,L]$ coordinate, and $(r,\theta)$ are polar coordinates on $D$.
\end{definition}

Some remarks on this definition: The Reeb vector field on the image of $\varphi$ is given by $\varphi_*\partial_t$. In particular, a thickened Reeb trajectory is a special kind of ``flow box'' for the Reeb vector field. Any embedded Reeb trajectory $[0,L]\to Y$ can be extended to a thickened Reeb trajectory for some $A>0$, by a Darboux-type argument similar to the proof of \cite[Thm.\ 2.5.1]{geiges}. The image of a thickened Reeb trajectory has contact volume $AL$, so necessarily
\[
AL \le \op{vol}(Y,\lambda).
\]

We can now ask if there are stronger upper bounds on $A$ in terms of $L$.

\begin{definition}
If $(Y,\lambda)$ is a closed contact three-manifold and $L>0$, define
\[
\op{Box}^L(Y,\lambda)\in\{-\infty\}\cup \left(0,\op{vol}(Y,\lambda)/L\right]
\]
to be the supremum of $A>0$ such that there exists a thickened Reeb trajectory in $(Y,\lambda)$ of area $A$ and length $L$.
\end{definition}

We have the following relation between thickened Reeb trajectories and quantitative closing lemmas, proved in \S\ref{sec:BoxClose} by a direct construction:

\begin{proposition}
\label{prop:BoxClose}
$\op{Box}^L(Y,\lambda) \le \op{Close}^L(Y,\lambda)$.
\end{proposition}

Note that if $\op{Close}^L(Y,\lambda)=0$, then by Remark~\ref{rem:closezeroinfinity}, there does not exist any embedded Reeb trajectory $[0,L]\to Y$, so by definition $\op{Box}^L(Y,\lambda)=-\infty$.

\subsection{Alternative spectral invariants}
\label{sec:ase}

In \cite{altech}, we defined a sequence of symplectic capacities for four-dimensional symplectic manifolds, which are an ``elementary'' alternative to ECH capacities, and which have similar properties and applications to four-dimensional symplectic embedding questions. We denote these alternative capacities here by $c_k^{\op{Alt}}$. They are indexed by a nonnegative integer $k$ and take values in $[0,\infty]$.

In the present paper we use related ideas to define an elementary alternative to the ECH spectrum of a contact three-manifold. For a closed contact three-manifold $(Y,\lambda)$ and a nonnegative integer $k$, we will define a number
\[
c_k(Y,\lambda)\in[0,\infty].
\]
These invariants have the basic following properties, most of which are analogous to properties of $c_k^{\op{Alt}}$ described in \cite[Thm.\ 6]{altech}. To state the properties, define an {\em orbit set\/} in $(Y,\lambda)$ to be a finite set of pairs $\alpha=\{(\alpha_i,m_i)\}$ where the $\alpha_i$ are distinct simple Reeb orbits, and the $m_i$ are positive integers. Define the symplectic action of the orbit set $\alpha$ by
\[
\mc{A}(\alpha) = \sum_im_i\mc{A}(\alpha_i).
\]

\begin{theorem}[proved in \S\ref{sec:pop}]
\label{thm:basicproperties}
The invariants $c_k$ have the following properties:
\begin{description}
\item{(Conformality)}
If $r>0$ then
\[
c_k(Y,r\lambda) = r c_k(Y,\lambda).
\]
\item{(Increasing)}
\[
0 = c_0(Y,\lambda) < c_1(Y,\lambda) \le c_2(Y,\lambda) \le \cdots \le +\infty.
\]
\item{(Disjoint Union)} If $(Y_i,\lambda_i)$ is a closed contact three-manifold for $i=1,\ldots,m$, then
\[
c_k\left(\coprod_{i=1}^m(Y_i,\lambda_i)\right) = \max_{k_1+\cdots+k_m=k} \sum_{i=1}^m c_{k_i}(Y_i,\lambda_i).
\]
\item{(Sublinearity)}
\[
c_{k+l}(Y,\lambda) \le c_k(Y,\lambda) + c_l(Y,\lambda).
\]
\item{(Monotonicity)}
If $f:Y\to\R^{\ge 0}$ then
\[
c_k(Y,\lambda) \le c_k(Y,e^f\lambda).
\]
\item{($C^0$-Continuity)}
For fixed $(Y,\lambda)$ and fixed $k$, the map $C^\infty(Y;\R)\to \R$ sending $f\mapsto c_k\left(Y,e^f\lambda\right)$ is $C^0$-continuous.
\item{(Spectrality)}
For given $(Y,\lambda)$ and $k$, if $c_k(Y,\lambda)$ is finite, then there exists an orbit set $\alpha$ such that $c_k(Y,\lambda)=\mc{A}(\alpha)$.
\item{(Liouville Domains)}
If $(Y,\lambda)$ is the boundary of a Liouville domain $(X,\omega$), see Definition~\ref{def:exact}, then
\[
c_k(Y,\lambda) \le c_k^{\op{Alt}}(X,\omega).
\] 
\item{(Sphere)}
\begin{equation}
\label{eqn:ckball}
c_k\left(\partial B^4(a)\right) = da,
\end{equation}
where $d$ is the unique nonnegative integer such that
\begin{equation}
\label{eqn:d}
d^2+d \le 2k \le d^2+3d.
\end{equation}
\item{(Asymptotic Lower Bound)}
\begin{equation}
\label{eqn:alb}
\liminf_{k\to\infty} \frac{c_k(Y,\lambda)^2}{k} \ge 2\op{vol}(Y,\lambda).
\end{equation}
\item{(Spectral Gap Closing Bound)}
If $k>0$ and $c_k(Y,\lambda) \le L < \infty$, then
\[
\op{Close}^L(Y,\lambda) \le c_k(Y,\lambda) - c_{k-1}(Y,\lambda).
\]
\end{description}
\end{theorem}

\begin{remark}
It follows from Theorem*~\ref{theoremstar:asymptotics} below that in fact, $c_k(Y,\lambda)$ is always finite. By the Spectrality property, an independent proof that $c_1(Y,\lambda)$ is always finite without using Seiberg-Witten theory would constitute a new proof of the Weinstein conjecture in three dimensions\footnote{The Weinstein conjecture in three dimensions asserts that every contact form on a closed three-manifold has a Reeb orbit. This was proved by Taubes \cite{taubesweinstein} using Seiberg-Witten theory. Proofs without using Seiberg-Witten theory are known in some cases; see e.g.\ the survey \cite{tw}.}.
\end{remark}

\begin{remark}
\label{rem:toric}
One can use the properties in Theorem~\ref{thm:basicproperties} to compute $c_k$ for many more examples than just $\partial B^4(a)$. For example, if $Y$ is the boundary of a ``convex toric domain'' or ``concave toric domain'' $X\subset \R^4$, then combinatorial formulas for $c_k^{\op{Alt}}(X)$ are given in  \cite[Thms.\ 9, 15]{altech}; and one can modify the proofs of these theorems to show that $c_k(Y)$ agrees with $c_k^{\op{Alt}}(X)$ in these cases.
\end{remark}

To work with the Spectral Gap Closing Bound, the following definition is useful:

\begin{definition}
\label{def:gap}
If $L>0$, define the {\em spectral gap\/} by
\[
\op{Gap}^L(Y,\lambda) = \inf_{k>0, \; c_k(Y,\lambda)\le L}\left(c_k(Y,\lambda) - c_{k-1}(Y,\lambda)\right).
\]
\end{definition}

The Spectral Gap Closing Property can then be restated as
\begin{equation}
\label{eqn:CloseGap}
\boxed{
\op{Close}^L(Y,\lambda) \le \op{Gap}^L(Y,\lambda).
}
\end{equation}

\begin{example}
If $Y=\partial B^4(a)$, then it follows from \eqref{eqn:ckball} that $c_1(Y)=c_2(Y)=a$. Thus
\[
\op{Gap}^L(Y) = \left\{\begin{array}{cl} +\infty, & L<a,\\ 0, & L\ge a. \end{array}\right.
\]
In this case equality holds in \eqref{eqn:CloseGap}, by Remark~\ref{rem:closezeroinfinity}, since every point in $\partial B^4(a)$ is on a simple Reeb orbit of action $a$.
\end{example}

See \S\ref{sec:starshaped} for a simple application of the bound \eqref{eqn:CloseGap} to prove Theorem~\ref{thm:ballgap} regarding star-shaped hypersurfaces.

With the help of Seiberg-Witten theory, we will prove in \S\ref{sec:asymptotics} that the alternative spectrum satisfies the following additional property:

\begin{theoremstar}[Asymptotics]
\label{theoremstar:asymptotics}
Let $(Y,\lambda)$ be any closed contact three-manifold. Then
\[
\lim_{k\to\infty} \frac{c_k(Y,\lambda)^2}{k} = 2\op{vol}(Y,\lambda).
\]
\end{theoremstar}

From this and \eqref{eqn:CloseGap} one can deduce the general quantitative closing lemma, Theorem*~\ref{theoremstar:qcl}, by a bit of calculus; see \S\ref{sec:qcl}.

\subsection{Relation with other work}

{\bf Alternative spectral invariants.\/}
As noted in \S\ref{sec:ase}, the definition of the alternate ECH spectrum given here is a variant of the definition of alternative ECH capacities of four-dimensional symplectic manifolds in \cite{altech}. The latter was inspired by a max-min invariant defined by McDuff-Siegel \cite{ms} for symplectic manifolds of any dimension as an alternative to Siegel's rational SFT capacities \cite{siegel}. Some similar elementary spectral invariants are defined by Edtmair \cite{oliver} in the context of area-preserving diffeomorphisms of closed surfaces, with the slight difference that these are ``relative'' invariants of a Hamiltonian isotopy, by contrast with our spectral invariants which are absolute invariants of a single contact form. After the first version of this paper appeared, a more general theory of elementary spectral gaps was developed in \cite{ct}.

{\bf Closing lemmas.\/} As noted in \S\ref{sec:qclintro}, the quantitative closing lemma in Theorem*~\ref{theoremstar:qcl} is a refinement of the ``strong closing property'' proved by Irie \cite{irie}. The basic mechanism used here for obtaining quantitative closing lemmas from spectral gaps is similar to the method used in \cite{pfh4} to obtain quantitative closing lemmas for area-preserving diffeomorphisms of closed surfaces in ``rational'' Hamiltonian isotopy classes. An statement similar to the strong closing property in this case was proved independently in \cite{cpz}.

Irie asked in \cite{iriestrong} whether the strong closing property holds for at least one higher dimensional contact manifold where the Reeb orbits are not already dense. This was proved for boundaries of ellipsoids in \cite{cdpt} using spectral gaps in contact homology. A different proof for this example was given using KAM theory by Xue \cite{xue}.

{\bf ECH spectrum.\/} Theorem**~\ref{theoremstarstar:comparison} below gives a relation between the invariants $c_k$ introduced here and the ECH spectrum originally introduced in \cite{qech}. The specific applications in this paper, namely Theorem*~\ref{theoremstar:qcl}, Theorem~\ref{thm:ballgap}, and Theorem~\ref{thm:ellipsoid}, can also be proved directly from the ECH spectrum, although for the latter two results this would require using Seiberg-Witten theory which we do not need here. See \S\ref{sec:othergaps} for a comparison of approaches.

\subsection{Outline of the rest of the paper}

In \S\ref{sec:preliminaries} we introduce some basic definitions that we need. In \S\ref{sec:maxmin} we define a family of ``spectral invariants'' $b_k$ of four-dimensional strong symplectic cobordisms, which is a precursor to the definition of the spectral invariants $c_k$ of contact three-manifolds. In \S\ref{sec:main} we define the invariants $c_k$ and prove Theorem~\ref{thm:basicproperties} describing their basic properties. In \S\ref{sec:calculus}, we use various calculations to prove Theorem~\ref{thm:ballgap}, Theorem~\ref{thm:ellipsoid}, and Proposition~\ref{prop:BoxClose}. Finally, in \S\ref{sec:comparison} we establish a relation between $c_k$ and the ECH spectrum; we use this to prove Theorem*~\ref{theoremstar:asymptotics} on the asymptotics of $c_k$ and then deduce the general quantitative closing lemma, Theorem*~\ref{theoremstar:qcl}.

\paragraph{Acknowledgments.} Thanks to Oliver Edtmair for helpful discussions.

%%%%%%%%%%%%%%%%%%%%%%%%%%%%%%%%%%%%%%%%%%%%

\section{Preliminaries}
\label{sec:preliminaries}

To prepare for the definition of alternative spectral invariants, we need the following mostly standard preliminaries.

\subsection{Symplectic cobordisms}

Let $(Y_+,\lambda_+)$ and $(Y_-,\lambda_-)$ be closed contact three-manifolds.

\begin{definition}
A {\em strong symplectic cobordism\/} from $(Y_+,\lambda_+)$ to $(Y_-,\lambda_-)$ is a compact symplectic four-manifold $(X,\omega)$ with boundary $\partial X = Y_+ - Y_-$, such that $\omega|_{Y_\pm}=d\lambda_\pm$.
\end{definition}

\begin{definition}
\label{def:exact}
An {\em exact symplectic cobordism\/} is a strong symplectic cobordism $(X,\omega)$ as above, for which there exists a $1$-form $\lambda$ on $X$ such that $d\lambda = \omega$ and $\lambda|_{Y_\pm} = \lambda_\pm$. A {\em Liouville domain\/} is an exact symplectic cobordism for which $Y_-=\emptyset$.
\end{definition}

\begin{example}
If $(Y,\lambda)$ is a closed contact three-manifold and $a<b$, then the {\em trivial cobordism\/}
\[
(X,\omega) = \left([a,b]\times Y,d\left(e^s\lambda\right)\right) 
\]
is an exact symplectic cobordism from $(Y,e^{b}\lambda)$ to $(Y,e^{a}\lambda)$. Here $s$ denotes the $[a,b]$ coordinate. When there is no chance for confusion, we will drop the symplectic form $d(e^s\lambda)$ from the notation.
\end{example}

Given an exact symplectic cobordism and a choice of $\lambda$ as in Definition~\ref{def:exact}, let $V$ denote the vector field on $X$ characterized by $\imath_V\omega = \lambda$. This is a Liouville vector field, i.e.\ the Lie derivative $\mc{L}_V\omega = \omega$, and it is transverse to the boundary of $X$. It follows that the flow of $V$ defines an identification
\begin{equation}
\label{eqn:N+}
N_+ \simeq (-\epsilon,0]\times Y_+
\end{equation}
where $N_+$ is a neighborhood of $Y_+$ in $X$, and $\lambda$ is identified with $e^s\lambda_+$, where $s$ denotes the $(-\epsilon,0]$ coordinate. Likewise we have an identification
\begin{equation}
\label{eqn:N-}
N_- \simeq [0,\epsilon)\times Y_-
\end{equation}
where $N_-$ is a neighborhood of $Y_-$ in $X$, and $\lambda$ is identified with $e^s\lambda_-$.

If $(X,\omega)$ is a strong symplectic cobordism which is not necessarily exact, we can still find a (noncanonical) Liouville vector field in a neighborhood of the boundary which is transverse to the boundary. This gives neighborhood identifications \eqref{eqn:N+} and \eqref{eqn:N-} identifying $\omega$ with $d(e^s\lambda_\pm)$.

In various constructions below, we implicitly assume that neighborhood identifications \eqref{eqn:N+} and \eqref{eqn:N-} as above have been chosen, and in the exact case that moreover $\lambda$ has been chosen. The spectral invariants that we define will ultimately be independent of these choices.

\begin{definition}
If $(X,\omega)$ is a strong symplectic cobordism from $(Y_+,\lambda_+)$ to $(Y_-,\lambda_-)$, define the {\em symplectization completion\/}
\[
\overline{X} = \left((-\infty,0]\times Y_-\right) \cup_{Y_-} X \cup_{Y_+} \left([0,\infty)\times Y_+\right).
\]
This is glued to a smooth manifold using the identifications \eqref{eqn:N+} and \eqref{eqn:N-}. There is a natural symplectic form $\overline{\omega}$ on $\overline{X}$, which agrees with $\omega$ on $X$ and with $d(e^s\lambda_{\pm})$ on the other two pieces.
\end{definition}

\begin{definition}
\label{def:composition}
If $(X_+,\omega)$ is a strong symplectic cobordism from $(Y_+,\lambda_+)$ to $(Y_0,\lambda_0)$, and if $(X_-,\omega_-)$ is a strong symplectic cobordism from $(Y_0,\lambda_0)$ to $(Y_-,\lambda_-)$, define the {\em composition\/}
\[
(X_-,\omega_-) \circ (X_+,\omega_+) = X_- \cup_{Y_0} X_+
\]
with the symplectic form that restricts to $\omega_\pm$ on $X_\pm$.
\end{definition}

Next let $(X,\omega)$ be a strong symplectic cobordism from $(Y_+,\lambda_+)$ to $(Y_-,\lambda_-)$, let $\Sigma$ be a compact oriented smooth surface with boundary, and let $f:\Sigma\to X$ be a smooth map such that $f(\partial\Sigma)\subset\partial X$. Write $\partial\Sigma = \partial_+\Sigma - \partial_-\Sigma$ where $\partial_\pm\Sigma$ denotes the portion of $\partial\Sigma$ that maps to $Y_\pm$, oriented with the $\pm$ boundary orientation of $\Sigma$. Let $A\in H_2(X,\partial X)$ denote the relative homology class represented by $f$.

\begin{lemma}
There is a homomorphism $\rho: H_2(X,\partial X)\to \R$ characterized by the property that if $f:\Sigma\to X$ represents $A\in H_2(X,\partial X)$ as above, then
\begin{equation}
\label{eqn:rho}
\int_\Sigma\omega - \int_{\partial_+\Sigma}\lambda_+ + \int_{\partial_-\Sigma}\lambda_- = \rho(A).
\end{equation}
\end{lemma}

\begin{proof}
We need to check that the left hand side of \eqref{eqn:rho} depends only on $A$. This is an exercise using Stokes's theorem.
\end{proof}

\begin{examples}
If $(X,\omega)$ is exact, then $\rho\equiv 0$, by Stokes's theorem.

If $Y_+=Y_-=\emptyset$, so that $(X,\omega)$ is a closed symplectic four-manifold, then $\rho$ is just the map $\langle[\omega],\cdot\rangle : H_2(X)\to\R$.

The map $\rho$ for a composition as in Definition~\ref{def:composition} is given by
\begin{equation}
\label{eqn:rhocomposition}
\rho(A) = \rho_+(A_-) + \rho_-(A_+)
\end{equation}
for $A\in H_2(X_-\circ X_+,\partial(X_-\circ X_+))$, where $A_\pm$ denotes the restriction of $A$ to $H_2(X_\pm,\partial X_\pm)$.
\end{examples}

\subsection{Almost complex structures}

\begin{definition}
Let $(Y,\lambda)$ be a closed contact three-manifold. An almost complex structure $J$ on $\R\times Y$ is {\em $\lambda$-compatible\/} if:
\begin{itemize}
\item $J(\partial_s)=R$, where $s$ denotes the $\R$ coordinate and $R$ denotes the Reeb vector field.
\item $J(\xi)=\xi$, rotating positively with respect to $d\lambda$, where $\xi=\Ker(\lambda)$.
\item $J$ is invariant under translation of the $\R$ factor on $\R\times Y$.
\end{itemize}
Let $\mc{J}(Y,\lambda)$ denote the space of $\lambda$-compatible almost complex structures on $\R\times Y$.
\end{definition}

\begin{definition}
Let $(X,\omega)$ be a strong symplectic cobordism from $(Y_+,\lambda_+)$ to $(Y_-,\lambda_-)$. An almost complex structure $J$ on the symplectization completion $\overline{X}$ is {\em cobordism-compatible\/} if:
\begin{itemize}
\item $J$ is $\omega$-compatible on $X$.
\item There are almost complex structures $J_\pm\in\mc{J}(Y_\pm,\lambda_\pm)$ such that $J$ agrees with the restriction of $J_-$ on $(-\infty,0]\times Y_-$ and with the restriction of $J_+$ on $[0,\infty)\times Y_+$.
\end{itemize}
Let $\mc{J}(\overline{X})$ denote the set of cobordism-compatible almost complex structures on $\overline{X}$.
\end{definition}

\subsection{Holomorphic curves}

Let $(X,\omega)$ be a strong symplectic cobordism from $(Y_+,\lambda_+)$ to $(Y_-,\lambda_-)$. Assume that the contact forms $\lambda_\pm$ are nondegenerate. Let $J\in\mc{J}(\overline{X})$.

\begin{definition}
Let $\mc{M}^J(\overline{X})$ denote the set of holomorphic maps
\[
u:(\Sigma,j) \longrightarrow (\overline{X},J),
\]
modulo reparametrization by biholomorphic maps $(\Sigma',j')\stackrel{\simeq}{\to}(\Sigma,j)$, such that:
\begin{itemize}
\item
The domain $(\Sigma,j)$ is a punctured compact Riemann surface, possibly disconnected.
\item
The map $u$ is nonconstant on each component of the domain $\Sigma$.
\item
For each puncture in $\Sigma$, either there is a Reeb orbit $\gamma$ in $Y_+$ and a neighborhood of the puncture mapping asymptotically to $[0,\infty)\times\gamma$ as $s\to\infty$ (we call this a {\em positive puncture\/}), or there is a Reeb orbit $\gamma$ in $Y_-$ and a neighborhood of the puncture mapping asymptotically to $(-\infty,0]\times\gamma$ as $s\to -\infty$ (we call this a {\em negative puncture\/}).
\end{itemize}
\end{definition}

Given $u\in\mc{M}^J(\overline{X})$, the positive punctures determine an orbit set $\alpha_+$ in $Y_+$ as follows: A pair $(\gamma,m)$ is an element of $\alpha_+$ when $u$ has at least one positive puncture asymptotic to a cover of $\gamma$, and $m$ is the sum over such punctures of the covering multiplicity of $\gamma$. Likewise, the negative punctures determine an orbit set $\alpha_-$ in $Y_-$. We denote by $\mc{M}^J(\overline{X},\alpha_+,\alpha_-)$ the set of $u\in\mc{M}^J(\overline{X})$ with corresponding orbit sets $\alpha_+$ and $\alpha_-$.

Note also that each $u\in\mc{M}^J(\overline{X})$ determines a well-defined relative homology class
\[
[u]\in H_2(X,\partial X).
\]

\begin{definition}
If $u\in\mc{M}^J(\overline{X},\alpha_+,\alpha_-)$, define the {\em energy\/}
\[
\mc{E}(u) = \mc{A}(\alpha_+) - \mc{A}(\alpha_-) + \rho([u])\in\R.
\]
Also define the {\em upper energy\/}
\[
\mc{E}_+(u) = \mc{E}(u) + \mc{A}(\alpha_-) = \mc{A}(\alpha_+) + \rho([u]). 
\]
\end{definition}

\begin{lemma}
\label{lem:energy}
If $u\in\mc{M}^J(\overline{X},\alpha_+,\alpha_-)$, then $\mc{E}(u) \ge 0$, with equality if and only if $u$ is the empty holomorphic curve (for which the domain is the empty set).
\end{lemma}

\begin{proof}
By perturbing if necessary, we can assume without loss of generality that $u$ is transverse to $\partial X$. By Stokes's theorem we have
\[
\mc{A}(\alpha_+) - \int_{u^{-1}(\{0\}\times Y_+)}\lambda_+ = \int_{u^{-1}([0,\infty)\times Y_+)}d\lambda_+.
\]
Similarly, we have
\[
\int_{u^{-1}(\{0\}\times Y_-)}\lambda_- - \mc{A}(\alpha_-)  = \int_{u^{-1}((-\infty,0]\times Y_-)}d\lambda_-.
\]
Finally, by \eqref{eqn:rho} we have
\[
\int_{u^{-1}(\{0\}\times Y_+)}\lambda_+ - \int_{u^{-1}(\{0\}\times Y_-)}\lambda_- + \rho([u]) = \int_{u^{-1}(X)}\omega .
\]
Adding the above three equations gives
\[
\mc{E}(u) =  \int_{u^{-1}([0,\infty)\times Y_+)}d\lambda_+ + \int_{u^{-1}(X)}\omega + \int_{u^{-1}((-\infty,0]\times Y_-)}d\lambda_-
\]
(which might be a more standard way to define ``energy''). Each of the integrands on the right hand side is pointwise nonnegative, by the definitions of $\lambda_\pm$-compatible and $\omega$-compatible almost complex structure. Thus $\mc{E}(u)\ge 0$.

If $\mathcal{E}(u)=0$, then $u^{-1}(X)=\emptyset$ (because $\omega$ is pointwise positive on $u^{-1}(X)$ wherever $du\neq 0$), and $u$ is everywhere tangent to the Reeb direction and the $\R$ direction outside of $X$ (by the definition of $\lambda_\pm$-compatible almost complex structure). This is impossible unless $u$ is empty.
\end{proof}

\begin{definition}
Let $(X,\omega)$ be a strong symplectic cobordism from $(Y_+,\lambda_+)$ to $(Y_-,\lambda_-)$ where $\lambda_+$ and $\lambda_-$ are nondegenerate, let $J\in\mc{J}(\overline{X})$, and let $x_1,\ldots,x_k\in\overline{X}$. Define $\mc{M}^J(\overline{X};x_1,\ldots,x_k)$ to be the set of $u\in\mc{M}^J(\overline{X})$ such that $x_1,\ldots,x_k\in\op{im}(u)$. Likewise, if $\alpha_\pm$ are orbit sets in $Y_\pm$, define $\mc{M}^J(\overline{X},\alpha_+,\alpha_-;x_1,\ldots,x_k)$ to be the set of $u\in\mc{M}^J(\overline{X},\alpha_+,\alpha_-)$ such that $x_1,\ldots,x_k\in\op{im}(u)$.
\end{definition}

%%%%%%%%%%%%%%%%%%%%%%%%%%%%%%%%%%%%%%%%%%%%%%%%%

\section{Max-min invariants of cobordisms}
\label{sec:maxmin}

As a precursor to the definition of the spectral invariant $c_k$ of contact three-manifolds, we now define a kind of spectral invariant $b_k$ of strong symplectic cobordisms, as a ``max-min'' of energy of holomorphic curves.

\begin{definition}
Let $(X,\omega)$ be a strong symplectic cobordism from $(Y_+,\lambda_+)$ to $(Y_-,\lambda_-)$ with $\lambda_+$ and $\lambda_-$ nondegenerate, and let $k$ be a nonnegative integer. Define
\begin{equation}
\label{eqn:defb}
b_k(X,\omega) = 
\sup_{\substack{J\in\mc{J}(\overline{X})\\ \mbox{\scriptsize $x_1,\ldots,x_k\in X$ distinct}}} \inf_{u\in\mc{M}^J(\overline{X};x_1,\ldots,x_k)} \mc{E}_+(u) \in [0,\infty].
\end{equation}
\end{definition}

\begin{lemma}[basic properties of $b_k$]
\label{lem:bkproperties}
Suppose below that $(X,\omega)$ is a strong symplectic cobordism from $(Y_+,\lambda_+)$ to $(Y_-,\lambda_-)$ with $\lambda_+$ and $\lambda_-$ nondegenerate.
\begin{description}
\item{(Conformality)}
If $r>0$ then
\[
b_k(X,r\omega) = r \cdot b_k(X,\omega).
\]
\item{(Increasing)}
\[
0 = b_0(X,\omega) < b_1(X,\omega) \le b_2(X,\omega) \le \cdots \le +\infty.
\]
\item{(Disjoint Union)} If $(X_i,\omega_i)$ is a strong symplectic cobordism between nondegenerate contact three-manifolds for $i=1,\ldots,m$, then
\[
b_k\left(\coprod_{i=1}^m(X_i,\omega_i)\right) = \max_{k_1+\cdots+k_m=k} \sum_{i=1}^m b_{k_i}(X_i,\omega_i).
\]
\item{(Sublinearity)}
\[
b_{k+l}(X,\omega) \le b_k(X,\omega) + b_l(X,\omega).
\]
\item{(Spectrality)}
If $\rho\equiv0$ on $H_2(X,\partial X$), e.g.\ if $(X,\omega)$ is exact, and if $b_k(X,\omega)$ is finite, then there exists an orbit set $\alpha_+$ in $Y_+$ with
\[
b_k(X,\omega) = \mc{A}(\alpha_+).
\]
\item{(Alternate Capacities)}
If $(X,\omega)$ is a Liouville domain, i.e.\ $(X,\omega)$ is exact and $Y_-=\emptyset$, then
\[
b_k(X,\omega) = c_k^{\op{Alt}}(X,\omega).
\]
\item{(Stripping\footnote{The idea of this terminology is that we can ``strip away'' part of a cobordism without increasing $b_k$.})}
If $(X_+,\omega)$ is a strong symplectic cobordism from $(Y_+,\lambda_+)$ to $(Y_0,\lambda_0)$, and if $(X_-,\omega_-)$ is a strong symplectic cobordism from $(Y_0,\lambda_0)$ to $(Y_-,\lambda_-)$, with $\lambda_+,\lambda_0,\lambda_-$ nondegenerate, then
\begin{equation}
\label{eqn:stripright}
b_k(X_-,\omega_-) \le b_k\left((X_-,\omega_-)\circ (X_+,\omega_+)\right).
\end{equation}
If in addition $\rho\equiv 0$ on $H_2(X_-,\partial X_-)$, e.g.\ the cobordism $(X_-,\omega_-)$ is exact, then
\begin{equation}
\label{eqn:stripleft}
b_k(X_+,\omega_+) \le b_k\left((X_-,\omega_-)\circ (X_+,\omega_+)\right).
\end{equation}
\end{description}
\end{lemma}

\begin{proof}
The Conformality property is immediate from the definition.

The Increasing property is immediate from the definition except for the assertion that $b_1(X,\omega)>0$, which follows from Lemma~\ref{lem:energy}.

The Disjoint Union and Sublinearity properties are immediate from the definition.

The Spectrality property holds because for any $u$ that arises in \eqref{eqn:defb}, the upper energy $\mc{E}_+(u)$, if finite, is the action of some orbit set; and by the nondegeneracy hypothesis on $\lambda_+$, there are only finitely many orbit sets below any given action level.

The Alternative Capacities property holds because in the special case when $(X,\omega)$ is a Liouville domain (with the contact form on the boundary nondegenerate), the definition of $b_k(X,\omega)$ is identical to the definition of $c_k^{\op{Alt}}(X,\omega)$ in \cite{altech}.

The nontrivial part of this lemma is the Stripping property. We begin with the first inequality \eqref{eqn:stripright}. Let $\epsilon>0$; it is enough to show that
\[
b_k(X_-) < b_k(X_-\circ X_+) + \epsilon.
\]
Choose $J_-\in\mc{J}(\overline{X_-})$ and $x_1,\ldots,x_k\in X_-$ distinct.  It is enough to show that there exists $u_-\in\mc{M}^{J_-}(\overline{X_-};x_1,\ldots,x_k)$ with
\begin{equation}
\label{eqn:Eplusuminus}
\mc{E}_+(u_-) \le b_k(X_-\circ X_+) + \epsilon.
\end{equation}

Since $J_-$ is cobordism-compatible, it agrees on $[0,\infty)\times Y_0$ with the restriction of an almost complex structure $J_0\in\mc{J}(Y_0,\lambda_0)$. Choose an almost complex structure $J_+\in\mc{J}(\overline{X_+})$ which agrees with $J_0$ on $(-\infty,0]\times Y_0$, and also in a neighborhood $N$ of $Y_0$ in $X_+$ identified with $([0,\delta)\times Y_0,d(e^s\lambda_0))$ for some $\delta>0$ as in \eqref{eqn:N-}.

For each positive integer $n$, we can choose an almost complex structure $J_n\in\mc{J}(\overline{X_-\circ X_+})$ such that:
\begin{itemize}
\item
$J_n$ agrees with $J_-$ on $((-\infty,0]\times Y_-)\cup_{Y_-}X_-$.
\item
$J_n$ agrees with $J_+$ on $(X_+\setminus N)\cup_{Y_+}([0,\infty)\times Y_+)$.
\item
There is a biholomorphism
\begin{equation}
\label{eqn:biholomorphism}
(N,J_n) \stackrel{\simeq}{\longrightarrow} ([0,n)\times Y_0,J_0).
\end{equation}
\end{itemize}

By the definition of $b_k$, for each $n$ we can find a holomorphic curve
\[
u_n\in\mc{M}^{J_n}(\overline{X_-\circ X_+},\alpha_+(n),\alpha_-(n);x_1,\ldots,x_k)
\]
for some orbit sets $\alpha_\pm(n)$ in $Y_\pm$ such that
\begin{equation}
\label{eqn:Eplusun}
\mc{E}_+(u_n) \le b_k(X_-\circ X_+) + \epsilon.
\end{equation}
It follows that we have an a priori energy bound
\[
\mc{E}(u_n) \le b_k(X_-\circ X_+) + \epsilon.
\]

By this energy bound and the Gromov compactness argument in \cite[Lem.\ 3]{altech}, we can pass to a subsequence such that:
\begin{itemize}
\item
$[u_n]\in H_2(X_-\circ X_+, \partial(X_-\circ\partial X_+))$ is independent of $n$; denote this class by $A$.
\item
The intersection of $u_n$ with $((-\infty,0]\times Y_-)\cup_{Y_-}X_-\cup_{Y_0}N$, regarded via the biholomorphism \eqref{eqn:biholomorphism} as a curve in  $((-\infty,0]\times Y_-)\cup_{Y_-}X_-\cup_{Y_0}([0,n)\times Y_0)$, converges as a current on compact sets to a curve
\[
u_-\in\mc{M}^{J_-}(\overline{X_-},\alpha_{0-},\alpha_-;x_1,\ldots,x_k),
\]
where $\alpha_{0-}$ is an orbit set in $Y_0$, and $\alpha_-$ is an orbit set in $Y_-$ with $\mc{A}(\alpha_-)\ge\mc{A}(\alpha_-(n))$ for each $n$. Let $A_-=[u_-]\in H_2(X_-,\partial X_-)$.
\item
The intersection of $u_n$ with $X_+\cup_{Y_+}([0,\infty)\times Y_+)$, regarded via the biholomorphism \eqref{eqn:biholomorphism} as a curve in $((-n+\delta,0]\times Y_0)\cup_{Y_0} X_+ \cup_{Y_+}([0,\infty)\times Y_+)$, converges as a current on compact sets to a curve
\[
u_+\in\mc{M}^{J_+}(\overline{X_+},\alpha_+,\alpha_{0+}),
\]
where $\alpha_+$ is an orbit set in $Y_+$ with $\mc{A}(\alpha_+(n)) \ge \mc{A}(\alpha_+)$ for each $n$, and $\alpha_{0+}$ is an orbit set in $Y_0$ with $\mc{A}(\alpha_{0+})\ge \mc{A}(\alpha_{0-})$. Let $A_+=[u_+]\in H_2(X_+,\partial X_+)$.
\item
$A=A_-+A_+$.
\end{itemize}

Using Lemma~\ref{lem:energy} and equation \eqref{eqn:rhocomposition}, we now compute that for any $n$,
\[
\begin{split}
\mc{E}_+(u_-) &= \rho(A_-) + \mc{A}(\alpha_{0-})\\
&\le \rho(A_-) + \mc{A}(\alpha_{0+})\\
&\le \rho(A_-) + \mc{A}(\alpha_{0+}) + \mc{E}(u_+)\\
&= \rho(A_-) + \rho(A_+) + \mc{A}(\alpha_+)\\
&= \rho(A) + \mc{A}(\alpha_+)\\
&\le \rho(A) + \mc{A}(\alpha_+(n))\\
&= \mc{E}_+(u_n)\\
&\le b_k(X_-\circ X_+) + \epsilon.
\end{split}
\]
Thus $u_-$ satisfies the desired property \eqref{eqn:Eplusuminus}.

Finally, to prove the inequality \eqref{eqn:stripleft}, let $\epsilon>0$, let $J_+\in\mc{J}(\overline{X_+})$, and let $x_1,\ldots,x_k\in X_+$ be distinct; it is enough to show that there exists $u_+\in\mc{M}^{J_+}(\overline{X_+})$ with
\begin{equation}
\label{eqn:Eplusuplus}
\mc{E}_+(u_+) \le b_k(X_-\circ X_+) + \epsilon.
\end{equation}
As above, for a suitable sequence of almost complex structures $J_n\in\mc{J}(\overline{X_-\circ X_+})$, we can find curves $u_n\in\mc{M}^{J_n}(\overline{X_-\circ X_+}; x_1,\ldots,x_k)$ satisfying \eqref{eqn:Eplusun}, and we can pass to a subsequence so that the analogues of the above four bullet points hold. Using the hypothesis that $\rho(A_-)=0$, we then have
\[
\begin{split}
\mc{E}_+(u_+) &= \rho(A_+) + \mc{A}(\alpha_+)\\
&= \rho(A) + \mc{A}(\alpha_+)\\
&\le \rho(A) + \mc{A}(\alpha_+(n))\\
&= \mc{E}_+(u_n)\\
&\le b_k(X_-\circ X_+)+\epsilon.
\end{split}
\]
Thus $u_+$ satisfies the desired property \eqref{eqn:Eplusuplus}.
\end{proof}

\begin{remark}
Although we do not need it in this paper, in the special case when $Y_-=\emptyset$, the Stripping property has the following implication. Writing $X=X_-$ and $X'=X_-\circ X_+$, the inequality \eqref{eqn:stripright} in this case can be restated as follows: If $(X,\omega)$ and $(X',\omega')$ are strong symplectic cobordisms with no negative boundary, and if there exists a symplectic embedding $(X,\omega) \to (\op{int}(X'),\omega')$, then $b_k(X,\omega) \le b_k(X',\omega')$. Thus $b_k$ can be used to define, for each $k$, a symplectic capacity\footnote{We have not quite proved that $b_k$ is a symplectic capacity, because we do not know if it monotone under symplectic embeddings $X\to X'$ that hit $\partial X'$. We would also want to remove the hypothesis that the contact form on the boundary is nondegenerate. One can fix these issues for example by definining the capacity of $X$ to be the supremum of $b_k(X')$ where $X'$ symplectically embeds into the interior of $X$ and the contact form on the boundary of $X'$ is nondegenerate.} for strong symplectic cobordisms with no negative boundary. If $X$ is not a Liouville domain, then it follows from the definition of $c_k^{\op{Alt}}$ in \cite{altech} and Stripping that $b_k(X) \ge c_k^{\op{Alt}}(X)$.
\end{remark}

%%%%%%%%%%%%%%%%%%%%%%%%%%%%%%%%%%%%%%%%%

\section{Definition and properties of alternative spectral invariants}
\label{sec:main}

We are now in a position to define the alternative spectral invariants of contact three-manifolds.

\subsection{The definition}

\begin{definition}
Let $(Y,\lambda)$ be a closed contact three-manifold, and let $k$ be a nonnegative integer. Define $c_k(Y,\lambda)\in[0,\infty]$ as follows.
\begin{itemize}
\item
If $\lambda$ is nondegenerate, define
\begin{equation}
\label{eqn:cknondeg}
c_k(Y,\lambda) = \sup_{R>0}b_k([-R,0]\times Y).
\end{equation}
\item
In the general case, define
\begin{equation}
\label{eqn:ckgeneral}
c_k(Y,\lambda) = \inf_{f:Y\to\R^{>0}}c_k\left(Y,e^f\lambda\right)
 = \sup_{f:Y\to\R^{<0}}c_k\left(Y,e^f\lambda\right),
 \end{equation}
where in the infimum and supremum we require that $e^f\lambda$ is nondegenerate. We will see in \S\ref{sec:pop} that the two definitions in \eqref{eqn:ckgeneral} agree, and also that they agree with \eqref{eqn:cknondeg} in the nondegenerate case.
\end{itemize}
\end{definition}

\begin{remark}
\label{rem:nondecreasing}
Note that $b_k([-R,0]\times Y)$ is a nondecreasing function of $R$, since increasing $R$ enlarges the set of almost complex structures and points over which we take the supremum in \eqref{eqn:defb}. Thus for any $R_0\ge 0$, instead of \eqref{eqn:cknondeg} we also have
\[
c_k(Y,\lambda) = \sup_{R>R_0}b_k([-R,0]\times Y).
\]
\end{remark}

\begin{remark}
One might be tempted to define $c_k(Y,\lambda)$ more simply as a max-min of positive energy of $J$-holomorphic curves in $\R\times Y$ for $J\in\mc{J}(Y,\lambda)$. However it is not clear if such a definition would satisfy some of the key properties of $c_k$ such as the Spectral Gap Closing Bound. In particular, almost complex structures in $\mc{J}(Y,\lambda)$ are not sufficiently flexible for the proof of Lemma~\ref{lem:gaplb} below to work. The definition \eqref{eqn:cknondeg} allows more general almost complex structures on $\R\times Y$, which do not satisfy the conditions for $\lambda$-compatibility in $[-R,0]\times Y$.
\end{remark}

\subsection{Proofs of properties}
\label{sec:pop}

\begin{lemma}
\label{lem:gaplb}
Suppose $f_1,f_2:Y\to\R$ are smooth functions with $f_1 < f_2$ such that the contact forms $e^{f_i}\lambda$ are nondegenerate. Let $(X,\omega)$ be a Liouville domain, and suppose there exists a symplectic embedding
\[
(X,\omega) \longrightarrow M_{f_1,f_2} \eqdef \{(s,y)\in\R\times Y \mid f_1(y) < s < f_2(y)\}.
\]
Then for any nonnegative integers $k$ and $l$, we have
\begin{equation}
\label{eqn:gaplb}
c_{k+l}\left(Y,e^{f_2}\lambda\right) \ge c_k\left(Y,e^{f_1}\lambda\right) + c_l^{\op{Alt}}(X,\omega).
\end{equation}
\end{lemma}

\begin{proof}
By \eqref{eqn:ckgeneral} and the $C^0$-continuity of $c_k^{\op{Alt}}$ proved in \cite{altech}, we can assume without loss of generality that the contact form on the boundary of $(X,\omega)$ is nondegenerate.

Choose $R_0>\max(f_2-f_1)$ and suppose that $R>0$. By the Stripping property \eqref{eqn:stripright}, applied to $X_-=M_{f_2-R_0-R,f_1}\sqcup X$ and $X_+=M_{f_1,f_2}\setminus\op{int}(X)$, we have 
\[
b_{k+l}\left(M_{f_2-R_0-R,f_2}\right) \ge b_{k+l}\left(M_{f_2-R_0-R,f_1}\sqcup X\right).
\]
By the Disjoint Union property in Lemma~\ref{lem:bkproperties}, we have
\[
b_{k+l}\left(M_{f_2-R_0-R,f_1}\sqcup X\right) 
\ge b_k(M_{f_2-R_0-R,f_1}) + b_l(X).
\]
By the Stripping property \eqref{eqn:stripleft} applied to $X_-=M_{f_2-R_0-R,f_1-R}$ and $X_+ = M_{f_1-R,f_1}$, we have
\[
b_k(M_{f_2-R_0-R,f_1}) 
\ge b_k(M_{f_1-R,f_1}) .
\]
By the Alternative Capacities property in Lemma~\ref{lem:bkproperties}, we have
\[
 b_l(X) = c_l^{\op{Alt}}(X).
\]
Combining the above four lines gives
\[
b_{k+l}\left(M_{f_2-R_0-R,f_2}\right) 
\ge
b_k(M_{f_1-R,f_1}) + c_l^{\op{Alt}}(X).
\]
Since $R>0$ was arbitrary, taking the supremum over $R>0$ and using Remark~\ref{rem:nondecreasing} proves \eqref{eqn:gaplb}.
\end{proof}

\begin{proof}[Proof of Theorem~\ref{thm:basicproperties}.]
Suppose to start that all contact forms under discussion are nondegenerate.

The Conformality, Increasing, Disjoint Union, Sublinearity, and Spectrality properties follow from the corresponding properties of $b_k$ in Lemma~\ref{lem:bkproperties}.

The Monotonicity property in the case $f>0$ follows from Lemma~\ref{lem:gaplb} with $X=\emptyset$ and $l=0$.

To prove the Liouville Domains property, let $V$ denote the Liouville vector field on $X$. For any $R>0$, the time $R$ flow of $V$ defines a symplectomorphism of a neighborhood $N$ of $Y$ in $X$ with $(-R,0]\times Y$. Then $X$ can be regarded as the composition of the exact cobordism $[-R,R]\times Y$ from $(Y,\lambda)$ to $(Y,e^{-R}\lambda)$, with the exact cobordism $X\setminus N$ from $(Y,e^{-R}\lambda)$ to the empty set. We now have
\[
b_k([-R,0]\times Y) \le b_k(X,\omega) = c_k^{\op{Alt}}(X,\omega),
\]
by the Stripping and Alternate Capacities properties in Lemma~\ref{lem:bkproperties}. Taking the supremum over $R$ gives $c_k(Y,\lambda) \le c_k^{\op{Alt}}(X,\omega)$.

It follows from Monotonicity and Conformality in the cases proved above that the definition \eqref{eqn:ckgeneral} is valid and agrees with the definition \eqref{eqn:cknondeg} in the nondegenerate case, and that the $C^0$-Continuity property holds. This continuity then implies that all the other properties we have proved so far also hold without assuming nondegeneracy, and that Monotonicity extends to the case $f\ge 0$. In addition, Lemma~\ref{lem:gaplb} extends to the case where $f_1\le f_2$ and $e^{f_i}\lambda$ may be degenerate.

To prove the Sphere property, recall from \cite{altech} that $c_k^{\op{Alt}}(B^4(a))=da$, where $d$ is given by \eqref{eqn:d}. Thus $c_k(\partial B^4(a))\le da$ by the Liouville Domains property which we just proved. The reverse inequality follows from a similar argument to the proof in \cite{altech} that $c_k^{\op{Alt}}(B^4(a))\ge da$.

To prove the Asymptotic Lower Bound property, we apply Lemma~\ref{lem:gaplb} where $(f_1,f_2)=(-R,0)$ for some $R>0$; $X$ is a disjoint union of balls; and $k=0$. By the calculation in \cite[\S3.2]{vc}, we obtain
\[
\limsup_{l\to\infty}\frac{c_l(Y,\lambda)^2}{l} \ge 2\left(1-e^{-R}\right)\op{vol}(Y,\lambda).
\]
Since $R>0$ can be taken arbitrarily large, the claim \eqref{eqn:alb} follows.

Finally, let us prove the Spectral Gap Closing Bound. Let $k>0$ and suppose that $c_k(Y,\lambda)\le L<\infty$. Suppose that $\op{Close}^L(Y,\lambda)>\delta$.  We need to show that
\begin{equation}
\label{eqn:biggap}
c_k(Y,\lambda) - c_{k-1}(Y,\lambda) \ge \delta.
\end{equation}

Since $\op{Close}^L(Y,\lambda)>\delta$, we can find a nonempty open set $\mc{U}\subset Y$, and a positive deformation $\{\lambda_\tau=e^{f_\tau}\lambda\}_{\tau\in[0,1]}$ of $\lambda$ supported in $\mc{U}$, such that $\op{width}(\{\lambda_\tau\})>\delta$, and the contact form $\lambda_\tau$ does not have any Reeb orbit intersecting $\mc{U}$ with action at most $L$ for any $\tau\in[0,1]$. By $C^0$-Continuity and Spectrality, the number $c_k(Y,\lambda_\tau)$ does not depend on $\tau$. Here we use the fact that the set of actions of orbit sets has measure zero, as explained in \cite{irie}. Thus
\begin{equation}
\label{eqn:combine1}
c_k(Y,\lambda) = c_k\left(Y,e^{f_1}\lambda\right).
\end{equation}
Since $\op{width}(\{\lambda_\tau\})>\delta$, we can find a symplectic embedding $B^4(\delta)\to M_{0,f_1}$. It then follows from Lemma~\ref{lem:gaplb} that
\begin{equation}
\label{eqn:combine2}
\begin{split}
c_k\left(Y,e^{f_1}\lambda\right) &\ge c_{k-1}(Y,\lambda) + c_1^{\op{Alt}}(B^4(\delta)) \\
&= c_{k-1}(Y,\lambda) + \delta.
\end{split}
\end{equation}
Combining \eqref{eqn:combine1} and \eqref{eqn:combine2} proves the desired inequality \eqref{eqn:biggap}.
\end{proof}

\begin{remark}
Generalizing the Monotonicity property, one might ask if the existence of an exact symplectic cobordism from $(Y_+,\lambda_+)$ to $(Y_-,\lambda_-)$ implies that $c_k(Y_-,\lambda_-)\le c_k(Y_+,\lambda_+)$. We do not know whether this is true.
\end{remark}

%%%%%%%%%%%%%%%%%%%%%%%%%%%%%%%%%%%%%%%%%%%%%%%%%%

\section{Some calculations}
\label{sec:calculus}

\subsection{Star-shaped domains}
\label{sec:starshaped}

\begin{proof}[Proof of Theorem~\ref{thm:ballgap}.]
Write $d=\floor{La^{-1}}$ and let $k=(d^2+3d)/2$. By Monotonicity and equation \eqref{eqn:ckball}, we have
\[
c_k(Y) \le c_k(B^4(a)) = da \le L. 
\]
By the Increasing property in Theorem~\ref{thm:basicproperties}, $c_i(Y)\le L$ whenever $i\le k$. Then by Definition~\ref{def:gap},
\[
\op{Gap}^L(Y) \le \min_{i=1,\ldots,k}(c_i(Y) - c_{i-1}(Y)).
\]
Thus
\[
\op{Gap}^L(Y) \le \frac{1}{k}\sum_{i=1}^k(c_i(Y) - c_{i-1}(Y)) = \frac{c_k(Y)}{k} \le \frac{da}{k} = \frac{2a}{d+3} = \frac{2a}{\floor{La^{-1}}+3}.
\]
We are now done by the inequality \eqref{eqn:CloseGap}.
\end{proof}

\subsection{Preliminaries on toric domains}
\label{sec:toricprelim}

Consider the moment map $\mu:\C^2\to\R^2_{\ge 0}$ defined by
\[
\mu(z) = \left(\pi|z_1|^2,\pi|z_2|^2\right).
\]
If $\Omega$ is a domain in $\R^2_{\ge 0}$, define the associated {\em toric domain\/}
\[
X_\Omega = \mu^{-1}(\Omega).
\]
A basic example is when $\Omega$ is the triangle with vertices at $(0,0)$, $(a,0)$, and $(0,b)$ for some $a,b>0$. In this case, $X_\Omega$ is the ellipsoid
\[
E(a,b) =\left\{z\in\C^2 \;\bigg|\; \frac{\pi|z_1|^2}{a}+ \frac{\pi|z_2|^2}{b} \le 1\right\}.
\]

More generally, suppose that the boundary of $\Omega$ consists of the line segment from $(0,0)$ to $(a,0)$ for some $a>0$, the line segment from $(0,0)$ to $(0,b)$ for some $b>0$, and a smooth curve $\partial_+\Omega$ from $(0,b)$ to $(a,0)$ which is transverse to the axes. Then $X_\Omega$ is a smooth domain in $\R^4$. If we further assume that $\partial_+\Omega$ is transverse to the radial vector field on $\R^2$, then $\partial X_\Omega$ is star-shaped. In this case the simple Reeb orbits on $\partial X_\Omega$ are described as follows:
\begin{itemize}
\item
$\mu^{-1}((a,0))$ is a simple Reeb orbit of action $a$.
\item
$\mu^{-1}((0,b))$ is a simple Reeb orbit of action $b$.
\item
Let $(x,y)\in\op{int}(\partial_+\Omega)$, and suppose that the outward normal vector to $\partial_+\Omega$ at $(x,y)$ is a multiple of $(m,n)$ where $m,n$ are relatively prime nonnegative integers. Then $\mu^{-1}((x,y))$ is foliated by simple Reeb orbits, each of which has action
\begin{equation}
\label{eqn:mxny}
\mc{A} = mx + ny.
\end{equation}
\end{itemize}
See e.g.\ \cite[\S2.2]{gh} for detailed calculations.

One more fact we need: For any domain $\Omega\subset\R^2_{\ge 0}$, suppose that $\Omega$ contains a triangle $T$ which is equivalent, by the action of an element of $\op{SL}_2(\Z)$ and/or a translation, to the triangle with vertices $(0,0)$, $(\delta,0)$, and $(0,\delta)$, for some $\delta>0$. Then by the ``Traynor trick'' \cite{traynor}, there exists a symplectic embedding $\op{int}(B^4(\delta)) \to \op{int}\left(X_\Omega\right)$. Hence the Gromov width $c_{\op{Gr}}\left(\op{int}\left(X_\Omega\right)
\right)\ge \delta$.

\subsection{Irrational ellipsoids}
\label{sec:ellipsoid}

\begin{proof}[Proof of Theorem~\ref{thm:ellipsoid}.]
By equation \eqref{eqn:CloseGap}, we need to show that
\begin{align}
\label{eqn:egap}
\op{Gap}^L(\partial E(a,1)) &\le \min(am_--n_-, \; n_+-am_+),\\
\label{eqn:eclose}
\op{Close}^L(\partial E(a,1)) &\ge \min(am_--n_-, \; n_+-am_+).
\end{align}

We begin with \eqref{eqn:egap}. An ellipsoid $E(a,b)$ is an example of both a ``convex toric domain'' and a ``concave toric domain''; see e.g.\ \cite{altech} for definitions. By Remark~\ref{rem:toric}, $c_k(\partial E(a,b))$ agrees with $c_k^{\op{Alt}}(E(a,b))$. By the combinatorial formula\footnote{Strictly speaking, the proof of \cite[Thm.\ 15]{altech} uses a bit of Seiberg-Witten theory. To avoid this, one can instead use the similar combinatorial formula for $c_k^{\op{Alt}}$ of a convex toric domain in \cite[Thm.\ 9]{altech} and a slight modification of the calculation in \cite[Ex.\ 1.23]{concave}.} for $c_k^{\op{Alt}}$ of a concave toric domain in \cite[Thm.\ 15]{altech} and the calculation in \cite[Ex.\ 1.23]{concave}, it follows that
\begin{equation}
\label{eqn:nab}
c_k(\partial E(a,b)) = N_k(a,b),
\end{equation}
where $(N_k(a,b))_{k\ge 0}$ denotes the sequence of nonnegative integer linear combinations of $a$ and $b$, listed in nondecreasing order with repetitions\footnote{This is also the sequence of ECH capacities of $E(a,b)$, see e.g.\ \cite{qech}, but we are avoiding using ECH here.}.

We are interested in the case when $a$ is irrational and $b=1$. Recall that $n_-/m_-<a$, $am_-\le L$, $n_+/m_+>a$, and $n_+\le L$. It follows that the numbers $am_-, am_+, n_-, n_+$ are all less than or equal to $L$; and by \eqref{eqn:nab} they all appear in the sequence $(c_k(\partial E(a,1))_{k\ge 0}$. Hence any positive difference between two of these numbers is greater than or equal to $\op{Gap}^L(\partial E(a,1))$. The inequality \eqref{eqn:egap} follows.

We now prove \eqref{eqn:eclose}. Let $\mc{U}$ denote the complement in $\partial E(a,1)$ of the two simple Reeb orbits $\mu^{-1}((a,0))$ and $\mu^{-1}((0,1))$. It is enough to show that for any $\delta$ less than the right hand side of \eqref{eqn:eclose}, there is a positive deformation $\{\lambda_\tau\}_{\tau\in[0,1]}$ of the contact form on $E(a,1)$ supported in $\mc{U}$, of width at least $\delta$, such that $\lambda_\tau$ does not have any Reeb orbit of action $\le L$ intersecting $\mc{U}$ for any $\tau\in[0,1]$.

Let $\Omega_0$ denote the triangle with vertices $(0,0)$, $(a,0)$, and $(0,1)$, so that $X_{\Omega_0}=E(a,1)$.  Let $L_1$ be the line through $(0,1)$ with slope $-m_+/n_+$. Let $L_2$ be the line through $(a,0)$ with slope $-m_-/n_-$. Since $L_1$ has slope greater than $-1/a$ and $L_2$ has slope less than $-1/a$, these lines intersect at a point $p$ in the first quadrant of the plane outside of $\Omega_0$. Let $T$ denote the triangle bounded by $L_1$, $L_2$, and the line segment from $(0,1)$ to $(a,0)$.

Consider a star-shaped toric domain $X_{\Omega}$ such that:
\begin{description}
\item{(*)}
The curve $\partial_+\Omega$ starts at $(0,1)$ and ends at $(a,0)$. Its slope starts at $-1/a$ and stays constant for a short time, then rapidly increases to an irrational number less than $-m_+/n_+$, then stays constant, then rapidly decreases to an irrational number larger than $-m_-/n_-$ then stays constant, then near its final endpoint rapidly increases back to $-1/a$ and stays constant.
\end{description}
Note that we necessarily have $\Omega_0\subset\Omega$ and $\Omega\setminus\Omega_0\subset T$. Also it is possible to choose $\Omega$ so that $\Omega\setminus\Omega_0$ is arbitrarily $C^0$-close to $T$.

For $\Omega$ satisfying condition (*) above, it follows using equation \eqref{eqn:mxny} that $\partial X_{\Omega}$ does not have any simple Reeb orbit of action less than or equal to $L$, except possibly for the two orbits $\mu^{-1}((0,1))$ and $\mu^{-1}((a,0))$. To show just one part of this calculation: Consider a simple Reeb orbit on $\partial X_{\Omega}$ arising from a point $(x,y)$ on the initial bend of $\partial_+\Omega$ near $(0,1)$ where the outward normal vector is proportional to $(m,n)$, where $m,n$ are relatively prime integers with $n>0$. Then the slope of $\partial_+\Omega$ at $(x,y)$ is $-m/n$. Since this slope satisfies $-1/a < -m/n < -m_+/n_+$, the minimality condition in the definition of $(m_+,n_+)$ implies that $n>L$. Then the action \eqref{eqn:mxny} is also greater than $L$ provided that $(x,y)$ is sufficiently close to $(0,1)$.

We can find a one-parameter family $\{\Omega_\tau\}_{\tau\in[0,1]}$ where $\Omega_0$ is the triangle specified above, $\Omega_\tau$ satisfies condition (*) above for each $\tau>0$, and $\Omega_1\setminus\Omega_0$ is $C^0$-close to $T$. For each $\tau$, there is a canonical contactomorphism
\[
\partial E(a,1) = \partial X_{\Omega_0} \stackrel{\simeq}{\longrightarrow} \partial X_{\Omega_\tau}
\]
defined by identifying pairs of points that are on the same ray in $\R^4$. Thus the family of contact manifolds $\{\partial X_{\Omega_\tau}\}$ is identified with a positive deformation $\{\lambda_\tau\}$ of $\partial E(a,1)$ supported in $\mc{U}$. Moreover, there is a canonical symplectomorphism
\[
M_{\lambda_1} \simeq \op{int}\left(X_{\Omega_1\setminus\Omega_0}\right).
\]
Since $\Omega_1\setminus\Omega_0$ can be chosen arbitrarily $C^0$-close to $T$, to complete the proof of the theorem, we just need to show that the Gromov width
\begin{equation}
\label{eqn:grt}
c_{\op{Gr}}(\op{int}(X_T)) \ge \min(am_- - n_-,\; n_+ - am_+).
\end{equation}

To prove \eqref{eqn:grt}, let $B$ denote the rectangle in the plane defined by $0\le x \le L/a$ and $0\le y \le L$. By definition, the points $(m_-,n_-)$ and $(m_+,n_+)$ are in $B$, but there are no lattice points in $B$ above the line $n_-x=m_-y$ and below the line $n_+x=m_+y$. In particular, the triangle with vertices $(0,0)$, $(m_-,n_-)$, and $(m_+,n_+)$ has no lattice points in its interior. Thus this triangle has area $1/2$ by Pick's theorem, so
\begin{equation}
\label{eqn:det}
m_-n_+ - m_+n_- = 1.
\end{equation}
This bit of number theoretic information is key for the rest of the calculation.

Using \eqref{eqn:det}, we compute that the vertex $p$ of the triangle $T$ is given by
\[
p = (n_+(am_--n_-), m_-(n_+-am_+)).
\]
By equation \eqref{eqn:det} again, the matrix
\[
A = \begin{pmatrix} -m_- & -n_- \\ -m_+ & -n_+ \end{pmatrix}
\]
is in $\op{SL}_2(\Z)$. Acting by $A$ on the triangle $T$, and then translating by $(am_-,n_+)$, yields the triangle with vertices
\[
(0,0), (am_--n_-,0), (0,n_+-am_+).
\]
Now the inequality \eqref{eqn:grt} follows by the Traynor trick, as reviewed in \S\ref{sec:toricprelim}. (In fact this is an equality by Gromov nonsqueezing \cite{gromov}.)
\end{proof}

\begin{remark}
Although we have shown that the closing bound \eqref{eqn:CloseGap} is sharp for the example of $\partial E(a,1)$, we do not expect that it is always sharp for contact three-manifolds. In particular, for the family of toric domains $X_{\Omega_\tau}$ constructed in the above proof, $\op{Gap}^L(\partial X_{\Omega_\tau})$ is independent of $\tau$, by the $C^0$-Continuity and Spectrality properties of $c_k$, as in the proof of the Spectral Gap Closing Bound. However it seems plausible that $\op{Close}^L(\partial X_{\Omega_\tau})$ gets smaller for large $\tau$.
\end{remark}

\subsection{More about thickened Reeb trajectories}
\label{sec:BoxClose}

\begin{proof}[Proof of Proposition~\ref{prop:BoxClose}.]
Let $\varphi:[0,L]\times D\to Y$ be a thickened Reeb trajectory of length $L$ and area $A$. Let $\epsilon>0$. It is enough to show that $\op{Close}^L(Y,\lambda)\ge A-\epsilon$. To do so, let $\mc{U}$ be the interior of the image of $\varphi$. We will show that there is a positive deformation $\{\lambda_\tau\}_{\tau\in[0,1]}$ of $\lambda$ supported in $\mc{U}$ such that:
\begin{description}
\item{(i)} For all $\tau\in[0,1]$, the contact form $\lambda_\tau$ does not have any Reeb orbit intersecting $\mc{U}$ with action $\le L$.
\item{(ii)} $\op{width}(\{\lambda_\tau\}) \ge A-\epsilon$.
\end{description}

To prepare for the construction of the positive deformation, fix a smooth function $\beta:[0,L]\to[0,1]$ such that $\beta(t)=0$ for $t$ close to $0$ or $L$, and $\beta(t)\equiv 1$ on a closed interval $I\subset (0,L)$ of length $L'\in(0,L)$.

Suppose that $g:[0,A/\pi]\to\R^{\ge 0}$ is a smooth function which does not vanish identically, such that $g'(x)\le 0$ and $g(x)=0$ for $x$ close to $A/\pi$. Consider the contact form $\lambda_g$ on $[0,L]\times D$ defined by
\[
\lambda_g = e^{\beta(t)g(r^2)}\left(dt + \frac{1}{2}r^2\,d\theta\right)
\]
where $t$ denotes the $[0,L]$ coordinate, and $(r,\theta)$ are polar coordinates on $D$. By construction, the contact form $\lambda_g$ agrees with $\varphi^*\lambda$ near the boundary of $[0,L]\times D$. We can now define a positive deformation $\{\lambda_\tau\}$ supported in $\mc{U}$ by defining $\lambda_\tau$ to agree with $\lambda$ outside of $\mc{U}$ and to satisfy $\varphi^*(\lambda_{\tau}) = \lambda_{\tau g}$. To complete the proof of the proposition, we want to show that $g$ can be chosen so that conditions (i) and (ii) above are satisfied.

A computation shows that the Reeb vector field associated to $\lambda_g$ is given by
\begin{equation}
\label{eqn:Rg}
R_g = e^{-\beta(t)g(r)}\left(
\left(1 + r^2 \beta(t) g'(r^2)\right)
\frac{\partial}{\partial t}
-
\frac{r\beta'(t)g(r^2)}{2}
\frac{\partial}{\partial r}
 - 2\beta(t)g'(r^2)
\frac{\partial}{\partial \theta}\right).
\end{equation}
Here only the $\partial/\partial t$ term is important for the discussion. In particular, henceforth let us require that $g$ satisfy the additional condition
\begin{equation}
\label{eqn:require}
1 + x g'(x) > 0.
\end{equation}
Then by \eqref{eqn:Rg}, the $\partial_t$ component of $R_g$ is always positive and less than or equal to $1$. Consequently:
\begin{itemize}
\item
$R_g$ has no periodic orbits in $[0,L]\times D$, and:
\item
Any trajectory of $R_g$ starting in $\{0\}\times D$ and ending in $\{L\}\times D$ has flow time at least $L$.
\end{itemize}
Since $\tau g$ also satisfies the requirements on $g$ for $\tau\in(0,1]$, it follows that condition (i) above is satisfied.

We now need to choose $g$ to also satisfy condition (ii) above.  To maximize the width of $\{\lambda_t\}$, we would like to choose $g$ as large as possible. The ``extreme case'' of \eqref{eqn:require} is where $1+xg'(x)=0$, so that
\[
e^{g(x)}= \frac{A}{\pi x}.
\]
Of course this is not allowed, since $g(0)$ needs to be finite and we need a strict equality in \eqref{eqn:require}, as well as $g(x)=0$ for $x$ close to $A/\pi$. However for any $\delta>0$, we can choose $g$ satisfying the requirements such that
\begin{equation}
\label{eqn:require2}
x > \delta \Longrightarrow e^{g(x)} > \frac{A-\epsilon}{\pi x}.
\end{equation}
We claim now that if $\delta>0$ is sufficiently small and \eqref{eqn:require2} holds, then condition (ii) is satisfied. More specifically we will show that if $\delta>0$ is sufficiently small and \eqref{eqn:require2} holds, then there is a symplectic embedding of $B^4(A-\epsilon)$ into $M_{\lambda_g}$.

The domain $M_{\lambda_g}$ contains a subset symplectomorphic to
\[
S = \left\{(s,t,r,\theta)\in\R\times I\times D^2 \;\big|\; (s,\varphi(t,r,\theta))\in M_{\lambda_g}\right\},
\]
with the symplectization symplectic form
\[
d\left(e^s\left(dt+\frac{1}{2}r^2\,d\theta\right)\right) = e^s\left(ds\,dt + \frac{1}{2}r^2\,ds\,d\theta + r\,dr\,d\theta\right).
\]
We will in fact show that if $\delta>0$ is sufficiently small and \eqref{eqn:require2} holds, then there is a symplectic embedding of $B^4(A-\epsilon)$ into $S$.

There is a symplectic embedding
\[
\psi: \R\times I\times D \longrightarrow \left(\C^2, r_1\,dr_1\,d\theta_1 + r_2\,dr_2\,d\theta_2\right)
\]
defined by
\[
\psi\left(s,t,r,\theta\right) = \left(re^{s/2+i\theta}, \sqrt{L/\pi}e^{s/2 + 2\pi i t/L}\right).
\]
The symplectomorphism $\psi$ maps the set $S$ to the set of $(r_1e^{i\theta_1},r_2e^{i\theta_2})\in\C^2$ such that:
\begin{itemize}
\item
$L\theta_2/(2\pi)\in I$.
\item
There exists $x\in[0,A/\pi]$ and $s\in(0,g(x))$ such that
\[
(\pi r_1^2,\pi r_2^2) = e^s (\pi r^2,  L).
\]
\end{itemize}
In particular, the set of possible values of $(\pi r_1^2, \pi r_2^2)$ is a union of line segments, where the set of line segments is parametrized by $x\in [0,A/\pi]$.  If \eqref{eqn:require2} holds, then this union of line segments includes all points in the rectangle
\[
[0,A-\epsilon] \times \left(L,\frac{L(A-\epsilon)}{\pi\delta}\right).
\]
Thus $S$ contains a region symplectomorphic to the symplectic Cartesian product of a closed disk of area $A-\epsilon$ and an open annulus of area
\[
L\left(\frac{A-\epsilon}{\pi\delta} - 1\right).
\]
If $\delta$ is small enough that the area of the annulus is greater than $A-\epsilon$, then $S$ contains a region symplectomorphic to the polydisk $P(A-\epsilon,A-\epsilon)$, i.e.\ the symplectic Cartesian product of two disks of area $A-\epsilon$, and the ball $B^4(A-\epsilon)$ is a subset of this polydisk.
\end{proof}

%%%%%%%%%%%%%%%%%%%%%%%%%%%%%%%%%%%%%%%%%%%%%%%%%%%%%%%%%%

\section{Comparison with the ECH spectrum}
\label{sec:comparison}

We now describe a relation between the invariants $c_k$ and the ECH spectrum, and we use this relation to prove Theorems*~\ref{theoremstar:asymptotics} and \ref{theoremstar:qcl}.

\subsection{Review of the ECH spectrum}

If the contact form $\lambda$ is nondegenerate, then the {\em embedded contact homology\/} of $(Y,\lambda)$, denoted by $ECH(Y,\lambda)$, is the homology of a chain complex over $\Z/2$ freely generated by orbit sets $\alpha=\{(\alpha_i,m_i)\}$ which are ``admissible'', meaning that $m_i=1$ whenever $\alpha_i$ is hyperbolic. Here a Reeb orbit is called ``hyperbolic'' when its linearized return map has real eigenvalues. The differential on the chain complex counts $J$-holomorphic curves in $\R\times Y$ with ECH index $1$ for a generic almost complex structure $J$ on $\R\times Y$. See \cite{bn} for a detailed exposition.

Taubes \cite{taubes} has shown that if $Y$ is connected\footnote{If $Y$ is disconnected, then the ECH of $(Y,\lambda)$ is the tensor product of the ECH of its components. It follows from Taubes's isomorphism applied to each component that the ECH of $(Y,\lambda)$ is still an invariant of $(Y,\xi)$.}, then there is a canonical isomorphism between $ECH(Y,\lambda)$ and a version of Seiberg-Witten Floer homology as defined by Kronheimer-Mrowka \cite{km}. It follows from this isomorphism that $ECH(Y,\lambda)$ depends only\footnote{In fact, at this level of description, $ECH$ depends only on $Y$. However certain additional structure on it, such as a direct sum decomposition $ECH(Y,\xi) = \oplus_{\Gamma\in H_1(Y)} ECH(Y,\xi,\Gamma)$, depends also on $\xi$.} on the contact structure $\xi=\Ker(\lambda)$, and so can be denoted by $ECH(Y,\xi)$.

If $Y$ is connected, then there is a map
\[
U: ECH(Y,\xi) \longrightarrow ECH(Y,\xi).
\]
This is induced by a chain map which counts $J$-holomorphic currents of ECH index 2 passing through a base point in $\R\times Y$. It is shown in \cite{taubes5} that under Taubes's isomorphism, this agrees with a counterpart in Seiberg-Witten Floer theory.

For $L\in\R$, the {\em filtered ECH\/}, denoted by $ECH^L(Y,\lambda)$, is defined to be the homology of the subcomplex spanned by admissible orbit sets with symplectic action less than $L$. It is shown in \cite{cc2} that this does not depend on the choice of $J$, although it does depend on the contact form $\lambda$ and not just the contact structure $\xi$. Inclusion of chain complexes induces a map
\begin{equation}
\label{eqn:imathl}
\imath_L: ECH^L(Y,\lambda) \longrightarrow ECH(Y,\lambda) = ECH(Y,\xi),
\end{equation}
and it is shown in \cite{cc2} that this map also does not depend on the choice of $J$. 

Given a nonzero class $\sigma \in ECH(Y,\xi)$, there is an associated {\em ECH spectral invariant\/}
\[
c^{\op{ECH}}_\sigma(Y,\lambda) \in \R,
\]
originally defined in \cite{qech}, which is the infimum of $L$ such that $\sigma$ is in the image of the map \eqref{eqn:imathl}. So far we have been assuming that $\lambda$ is nondegenerate, but the spectral invariant $c^{\op{ECH}}_\sigma$ is $C^0$-continuous and so canonically extends in a $C^0$-continuous manner to the degenerate case.

Given $Y,\lambda,\sigma$, there exists an orbit set $\alpha$ such that $c^{\op{ECH}}_\sigma(Y,\lambda) = \mc{A}(\alpha)$. This follows from the definition when $\lambda$ is nondegenerate, and by a compactness argument when $\lambda$ is degenerate.

\subsection{The comparison theorem}

\begin{theoremstarstar}
\label{theoremstarstar:comparison}
Let $(Y,\lambda)$ be a closed connected contact three-manifold and let $k>0$. Then
\begin{equation}
\label{eqn:comparison}
c_k(Y,\lambda) \le \inf\left\{c^{\op{ECH}}_\sigma(Y,\lambda) \mid U^k\sigma\neq 0\right\}.
\end{equation}
\end{theoremstarstar}

Note that for general $Y$, not all of the ECH spectral invariants appear on the right hand side of \eqref{eqn:comparison}; and the alternative spectral invariants $c_k(Y,\lambda)$ do not have to agree with any ECH spectral invariants.

\begin{proof}[Proof of Theorem**~\ref{theoremstarstar:comparison}.]
By $C^0$ continuity of the ECH spectrum, we can assume without loss of generality that $\lambda$ is nondegenerate.

To start, we discuss the map on ECH induced by the cobordism $[-R,0]\times Y$ for $R>0$. Let $J\in\mc{J}\left(\overline{[-R,0]\times Y}\right)$, and let $J_+$ and $J_-$ denote the induced almost complex structures on the symplectizations of $(Y,\lambda)$ and $(Y,e^{-R},\lambda)$. Suppose that $J_+$ and $J_-$ are generic so that we have well-defined ECH chain complexes $ECC(Y,\lambda,J_+)$ and $ECC(Y,e^{-R},J_-)$. By \cite[Thm.\ 1.9]{cc2}, if $L>0$, then the cobordism $[-R,0]\times Y$ induces a map
\[
\Phi: ECH^L(Y,\lambda) \longrightarrow ECH^L(Y,e^{-R},\lambda)
\]
such that
\begin{equation}
\label{eqn:imathls}
\imath_L\circ\Phi=\imath_L:ECH^L(Y,\lambda) \longrightarrow ECH(Y,\xi).
\end{equation}
Furthermore, similarly to the proof of \cite[Thm.\ 12]{altech}, for any distinct points $x_1,\ldots,x_k\in[-R,0]\times Y$, the map $\Phi\circ U^k$ is induced by a (noncanonical) chain map
\[
\phi: ECC^L(Y,\lambda,J_+) \longrightarrow ECC^L(Y,e^{-R}\lambda,J_-)
\]
with the following property:
\begin{description}
\item{(*)}
If $\alpha_+$ and $\alpha_-$ are admissible orbit sets such that the coefficient $\langle\phi\alpha_+,\alpha_-\rangle\neq 0$, then there exists an orbit set $\alpha_+'$ for $\lambda$ with $\mc{A}(\alpha_+)\ge \mc{A}(\alpha_+')$, an orbit set $\alpha_-'$ for $e^{-R}\lambda$ with $\mc{A}(\alpha_-')\ge \mc{A}(\alpha_-)$, and a holomorphic curve $u\in\mc{M}^J\left(\overline{[-R,0]\times Y},\alpha_+',\alpha_-';x_1,\ldots,x_k\right)$.
\end{description}

We now apply this setup to the matter at hand. Let $\sigma\in ECH(Y,\xi)$ and suppose that $U^k\sigma\neq 0$. We want to show that
\begin{equation}
\label{eqn:wwtst}
c_k(Y,\lambda) \le c^{\op{ECH}}_\sigma(Y,\lambda).
\end{equation}

Suppose that $c^{\op{ECH}}_\sigma(Y,\lambda) < L$. Let $J,x_1,\ldots,x_k,\phi$ be as above. Since $c^{\op{ECH}}_\sigma(Y,\lambda) < L$, the class $\sigma$ is represented by a cycle $\eta$ in $ECC^L(Y,\lambda,J_+)$. It follows from \eqref{eqn:imathls} that $\phi(\eta)$ represents the nonzero class $U^k\sigma$. Consequently there must exist generators $\alpha_+$ of $ECC^L(Y,\lambda,J_+)$ and $\alpha_-$ of $ECC^L(Y,e^{-R},J_-)$ with $\langle\phi\alpha_+,\alpha_-\rangle\neq 0$. By the property (*), we deduce that there exists $u\in \mc{M}^J\left(\overline{[-R,0]\times Y},\alpha_+',\alpha_-';x_1,\ldots,x_k\right)$ with $\mc{A}(\alpha_+')<L$. Since $L>c^{\op{ECH}}_\sigma(Y,\lambda)$ was arbitrary, we deduce that
\[
\inf_{u\in\mc{M}^J\left(\overline{[-R,0]\times Y};x_1,\ldots,x_k\right)}\mc{E}_+(u) \le c^{\op{ECH}}_\sigma(Y,\lambda).
\]

Since we started from any $J\in\mc{J}\left(\overline{[-R,0]\times Y}\right)$ with $J_\pm$ generic and any $x_1,\ldots,x_k\in[-R,0]\times Y$ distinct, it follows that
\[
\sup_{\substack{{\mbox{\scriptsize $J\in\mc{J}\left(\overline{[-R,0]\times Y}\right): J_\pm$ generic}}\\
 \mbox{\scriptsize $x_1,\ldots,x_k\in [-R,0]\times Y$ distinct}}}
\inf_{u\in\mc{M}^J\left(\overline{[-R,0]\times Y};x_1,\ldots,x_k\right)}\mc{E}_+(u) \le c^{\op{ECH}}_\sigma(Y,\lambda).
\]
The stipulation in the above inequality that $J_\pm$ are generic can be dropped by a Gromov compactness argument. We conclude that
\[
b_k([-R,0]\times Y) \le c^{\op{ECH}}_\sigma(Y,\lambda).
\]
Taking the supremum over $R>0$ proves the desired inequality \eqref{eqn:wwtst}.
\end{proof}

\subsection{Asymptotics of $c_k$}
\label{sec:asymptotics}

\begin{proof}[Proof of Theorem*~\ref{theoremstar:asymptotics}.]
As in \cite[Prop.\ 8.4]{qech}, we can assume without loss of generality that $Y$ is connected.

We recall that if $\Gamma\in H_1(Y)$, then $ECH(Y,\xi,\Gamma)$ is defined to be the homology of the subcomplex of the ECH chain complex generated by admissible orbit sets $\alpha=\{(\alpha_i,m_i)\}$ with total homology class $[\alpha]=\sum_im_i[\alpha_i]=\Gamma\in H_1(Y)$. One can choose a class $\Gamma$ such that $c_1(\xi) + 2\op{PD}(\Gamma)\in H^2(Y;\Z)$ is torsion. For such a class $\Gamma$, there is a relative $\Z$-grading on $ECH(Y,\xi,\Gamma)$; and it follows from results of Kronheimer-Mrowka \cite{km}, as explained for example in \cite[Lem.\ A.1]{infinity}, that in sufficiently large grading, $ECH(Y,\xi,\Gamma)$ is nonzero for gradings of at least one parity, and the $U$-map is an isomorphism. Consequently we can choose a ``$U$-sequence'', namely a sequence of nonzero homogeneous classes $\sigma_k\in ECH(Y,\xi,\Gamma)$ indexed by $k\ge 0$ such that $U\sigma_k = \sigma_{k-1}$ whenever $k>0$.

By \cite[Thm.\ 1.3]{vc}, if $\{\sigma_k\}$ is a $U$-sequence as above, then
\[
\lim_{k\to \infty}\frac{c^{\op{ECH}}_{\sigma_k}(Y,\lambda)^2}{k} = 2\op{vol}(Y,\lambda).
\]
By Theorem**~\ref{theoremstarstar:comparison}, we also know that
\[
c_k(Y,\lambda) \le c^{\op{ECH}}_{\sigma_k}(Y,\lambda).
\]
The above two lines imply that
\[
\limsup_{k\to\infty} \frac{c_{k}(Y,\lambda)^2}{k} \le 2\op{vol}(Y,\lambda).
\]
We are now done by the Asymptotic Lower Bound property of $c_k$.
\end{proof}

\subsection{Proof of the general quantitative closing lemma}
\label{sec:qcl}

\begin{proof}[Proof of Theorem*~\ref{theoremstar:qcl}.]
By scaling the contact form, we can assume without loss of generality that $\op{vol}(Y,\lambda)=1$. Below we write $c_k=c_k(Y,\lambda)$ and $\op{Gap}^L = \op{Gap}^L(Y,\lambda)$.

By the bound \eqref{eqn:CloseGap}, it is enough to show that
\begin{equation}
\label{eqn:limsup}
\limsup_{L\to\infty}\left(L\cdot\op{Gap}^L\right) \le 1.
\end{equation}
By Theorem*~\ref{theoremstar:asymptotics}, in proving \eqref{eqn:limsup}, we can restrict attention to numbers $L$ that are values of $c_k$. That is, it is enough to show that for every $\epsilon>0$, if $m$ is sufficiently large and $L=c_m$, then
\begin{equation}
\label{eqn:lgl}
L\cdot\op{Gap}^L \le 1+\epsilon.
\end{equation}

Let $\delta>0$. By Theorem~\ref{theoremstar:asymptotics}, if $m$ is sufficiently large, then
\begin{equation}
\label{eqn:msl}
k\ge m/2 \Longrightarrow
\left|c_k^2 - 2k\right| \le \delta k.
\end{equation}
Fix such $\delta$ and $m$. Let $n=am$ be an integer with $0<a<1/2$. Then for $L=c_m$, we have
\[
\begin{split}
\op{Gap}^L &\le \min_{i=1,\ldots,n}\left(c_{m-n+i} - c_{m-n+i-1}\right)\\
&\le \frac{1}{n}\left(c_m - c_{m-n}\right)\\
&\le \frac{1}{n}\left(\sqrt{(2+\delta)m} - \sqrt{(2-\delta)(m-n)}\right)\\
&= \frac{1}{a\sqrt{m}}\left(\sqrt{2+\delta} - \sqrt{(2-\delta)(1-a)}\right).
\end{split}
\]
Since $L\le\sqrt{(2+\delta)m}$, it follows that
\[
L\cdot \op{Gap}^L \le \frac{2+\delta-\sqrt{(4-\delta^2)(1-a)}}{a}.
\]

Let $f(a,\delta)$ denote the right hand side of the above inequality. For fixed $a>0$, we have
\[
\lim_{\delta\searrow 0} f(a,\delta) = \frac{2\left(1-\sqrt{1-a}\right)}{a}.
\]
Let $g(a)$ denote the right hand side of the above equation. Then we have
\[
\lim_{a\searrow 0}  g(a) = 1.
\]
Consequently, given $\epsilon>0$, to arrange \eqref{eqn:lgl}, we can first choose $a$ sufficiently small that $g(a)<1+\epsilon/3$, then choose $\delta$ sufficiently small that $f(a,\delta)<g(a)+\epsilon/3$, and then require that $m$ is sufficiently large that \eqref{eqn:msl} holds. Note that when $am$ is not an integer, we can fix this by slightly increasing $a$, and the $\epsilon/3$ of room that we have left will suffice when $m$ is sufficiently large.
\end{proof}

\subsection{Other quantitative closing bounds}
\label{sec:othergaps}

\begin{remark}
\label{rem:othergaps}
A similar argument to the proof of Theorem**~\ref{theoremstarstar:comparison} can be used to give a closing bound directly in terms of the ECH spectrum, namely
\begin{equation}
\label{eqn:ECHbound}
\op{Close}^L(Y,\lambda)\le\inf\left\{c^{\op{ECH}}_\sigma(Y,\lambda) - c^{\op{ECH}}_{U\sigma}(Y,\lambda) \;\big|\; U\sigma\neq0,\; c^{\op{ECH}}_\sigma(Y,\lambda)\le L\right\}.
\end{equation}
One can give alternate proofs of the general quantitative closing lemma in Theorem*~\ref{theoremstar:qcl}, as well as the more specific results in Theorem~\ref{thm:ballgap} and Theorem~\ref{thm:ellipsoid}, directly from the bound \eqref{eqn:ECHbound} together with known properties of the ECH spectrum.

In simple examples such as boundaries of convex or concave toric domains in $\R^4$, the bounds \eqref{eqn:CloseGap} and \eqref{eqn:ECHbound} are equivalent. In more general examples, in principle either bound might be stronger than the other. 
\end{remark}

%%%%%%%%%%%%%%%%%%%%%%%%%%%%%%%%%%%%%%%%%%%%%%%%%%%%

\end{document}